\documentclass[12pt,english,a4paper]{smfart}

\usepackage[T1]{fontenc}
\usepackage{lmodern}
\usepackage{smfthm}
\usepackage[headings]{fullpage}
\usepackage{amssymb}

\newcommand{\Hom}{\operatorname{Hom}}

\newcommand{\wotimes}{\operatorname{\widehat{\otimes}}}
\newcommand{\Gal}{\operatorname{Gal}}
\newcommand{\Qp}{\mathbf{Q}_p}
\newcommand{\Qpbar}{\overline{\mathbf{Q}}_p}
\newcommand{\Cp}{\mathbf{C}_p}
\newcommand{\QQ}{\mathbf{Q}}
\newcommand{\Gm}{\mathbf{G}_\mathrm{m}}
\newcommand{\calK}{\mathcal{K}}
\newcommand{\tph}{T_p H}
\newcommand{\tphx}{T_p^\times\! H}
\newcommand{\et}{\widetilde{\mathbf{E}}}
\newcommand{\at}{\widetilde{\mathbf{A}}}
\newcommand{\efont}{\mathbf{E}}
\newcommand{\afont}{\mathbf{A}}
\newcommand{\hatU}{\widehat{U}}
\newcommand{\OO}{\mathcal{O}}
\newcommand{\MM}{\mathfrak{m}}
\newcommand{\val}{\operatorname{val}}
\newcommand{\intrm}{\operatorname{int}}
\newcommand{\vp}{\val_p}
\newcommand{\Diff}{\operatorname{Diff}}
\newcommand{\pscal}[1]{\langle #1 \rangle}
\newcommand{\bigO}{\mathrm{O}}
\newcommand{\dcroc}[1]{[\![ #1 ]\!]}
\renewcommand{\phi}{\varphi}

\title{Galois measures and the Katz map}

\author{Laurent Berger}
\address{UMPA, ENS de Lyon \\ 
UMR 5669 du CNRS}
\email{laurent.berger@ens-lyon.fr}
\urladdr{https://perso.ens-lyon.fr/laurent.berger/}

\begin{document}

\begin{abstract}
The purpose of this paper is to explain the proofs of the results announced by Nick Katz in 1977, namely a description of ``Galois measures for Tate modules of height two formal groups over the ring of integers of a finite unramified extension of $\mathbf{Q}_p$''. 
\end{abstract}

\date{\today}



\maketitle

\tableofcontents

\setlength{\baselineskip}{18pt}

\section*{Introduction}
\label{intro}

The purpose of this paper is to explain the proofs of the results announced by Nick Katz in 1977, namely a theory of ``Galois measures for Tate modules of height two formal groups over the ring of integers of a finite unramified extension of $\Qp$''. As Katz writes in the introduction of \cite{Kat81}, these results were announced in \cite{Kat77} but ``the details of this general theory remain unpublished''.

After the first version of this paper was written, a far-reaching generalization of $p$-adic Fourier theory appeared in \cite{GHH25}. This generalization includes an integral Fourier transform for $p$-divisible groups. In Remark 7.2.4, Graham, van Hoften and Howe discuss the relation between Katz's results and their integral restriction map $r^{\intrm}$. They point out the interest of understanding its injectivity and surjectivity more systematically. The results proved in the present paper fit naturally into this recent framework, and we hope that our methods can be used to study their map $r^{\intrm}$.

In the author's joint paper \cite{AB24} with Konstantin Ardakov, Katz' results were proved in the case of a Lubin--Tate formal group attached to a uniformizer of $\QQ_{p^2}$ (Theorem 1.6.1 of \cite{AB24}). In the present paper, we  simplify and clarify that proof, and extend it to all height two formal groups. We now describe our results in more detail. 

\subsection*{Formal groups}
Let $K$ be a finite unramified extension of $\Qp$ with ring of integers $\OO_K$ and residue field $k$. Let $\Qpbar$ be an algebraic closure of $K$, let $G_K = \Gal(\Qpbar/K)$ and let $\Cp$ be the completion of $\Qpbar$. Let $G$ be a formal group of dimension $1$ and height $2$ over $\OO_K$. Let $A(G) = \OO_K \dcroc{X}$ denote the coordinate ring of $G$. If $S$ is a flat $\OO_K$-algebra, we let $S \wotimes A(G) = S \dcroc{X}$.

Let $H$ be the Cartier dual of $G$ (denoted by $G^{\vee}$ in \cite{Kat77}), so that (cf. \cite{Tat67}) $H$ is also a formal group of dimension $1$ and height $2$. The Tate module $\tph$ parameterizes all formal group homomorphisms $G \to \Gm$ over $\OO_{\Cp}$. They are given by power series $t(X) \in \OO_{\Cp} \wotimes A(G) = \OO_{\Cp} \dcroc{X}$ such that $t(0)=1$ and $t(X \oplus_G Y) = t(X) \cdot t(Y)$. 

\subsection*{The covariant bialgebra of $G$}
Let $U(G)$ be the covariant bialgebra of $G$ (we use $U(G)$ instead of Katz' algebra $\Diff(G)$ of translation-invariant differential operators on $G$. We have $U(G) \simeq \Diff(G)$, see Remark \ref{ugdiff}). The set $U(G)$ is the set of all $\OO_K$-linear maps $A(G) \to \OO_K$ that vanish on some power of the augmentation ideal. 
It can be equipped with an algebra structure induced by the addition on $G$. 

Let $\hatU(G)$ denote the set of $\OO_K$-linear maps $A(G) \to \OO_K$ that are continuous for the $(p,X)$-adic topology, so that $\hatU(G)$ is the $p$-adic completion of $U(G)$. Let \[ \pscal{\cdot,\cdot} : \hatU(G) \times A(G) \to \OO_K \] denote the evaluation pairing; it gives an isomorphism $\Hom_{\OO_K}(\hatU(G),\OO_K) = A(G)$.

\subsection*{The Katz map}
Let $C^0_{\Gal}(\tph, \OO_{\Cp})$ denote the $\OO_K$-module of Galois-continuous functions, namely those functions $f  : \tph \to \OO_{\Cp}$ that are continuous and $G_K$-equivariant: $f(\sigma(t)) = \sigma(f(t))$ for all $t \in \tph$ and $\sigma \in G_K$. 

The evaluation pairing extends to $\pscal{\cdot,\cdot} : \hatU(G) \times (\OO_{\Cp} \wotimes A(G)) \to \OO_{\Cp}$. If $u \in \hatU(G)$, then $t \mapsto \pscal{u, t(X)}$ is a Galois-continuous function $\tph \to \OO_{\Cp}$. The Katz map is
\[ \calK : \hatU(G) \to C^0_{\Gal}(\tph, \OO_{\Cp}), \]
defined by $\calK(u) (t) = \pscal{u, t(X)}$ (this is the map $(\ast)$ on page 59 of \cite{Kat77}). 

\subsection*{Galois measures}
Let $S$ be a $p$-adically complete and separated flat $\OO_K$-algebra. Applying the functor $\Hom_{\OO_K}(\cdot,S)$ to $\calK : \hatU(G) \to C^0_{\Gal}(\tph, \OO_{\Cp})$ gives an $S$-linear map (the map $(\ast\ast)$ on page 59 of \cite{Kat77})
\[ \calK^\ast : \Hom_{\OO_K} ( C^0_{\Gal}(\tph, \OO_{\Cp}), S) \to S \wotimes A(G). \]
If $f(X) \in S \wotimes A(G)$, let $\psi f \in S[1/p] \wotimes A(G)$ be defined by (compare with \S III of \cite{Col79}) 
\[ (\psi f)([p]_G(X)) = \frac{1}{p^2} \cdot \sum_{\omega \in G[p]} f(X \oplus_G \omega). \] 
We say that $f(X) \in S \wotimes A(G)$ is $\psi$-integral if $\psi^n f \in S \wotimes A(G)$ for all $n \geq 0$.
Let $\tphx = \tph \setminus p \cdot \tph$.

\subsection*{Main results}

The main result claimed by Katz (see page 60 of \cite{Kat77}) is the following:

\begin{enonce*}{Theorem A}
\label{theoA}
The map $\calK^\ast$ is injective, and its image is the set of $\psi$-integral power series $f(X) \in S \wotimes A(G)$. 

Moreover, a Galois measure $\mu$ is supported on $\tphx$ if and only if $\psi(\calK^\ast(\mu)) = 0$.
\end{enonce*}

The main result of this paper is a proof of Theorem B below. Choose an element $t = (t_0,t_1,\hdots) \in \tphx$ and let $\calK_t : \hatU(G) \to \OO_{\Cp}$ denote the map $u \mapsto \calK(u)(t)$. Let $K_\infty$ denote the completion of the field generated over $K$ by the $t_n$ for $n \geq 1$. Since $\calK(u)$ is $G_K$-equivariant, the Ax-Sen-Tate theorem implies that $\calK(u)(t) \in \OO_{K_\infty}$.

\begin{enonce*}{Theorem B}
\label{theoB}
The map $\calK$ is injective, and the map $\calK_t : \hatU(G) \to \OO_{K_\infty}$ is surjective.
\end{enonce*}

Theorem A follows from Theorem B by mostly formal arguments, that are carried out in detail in \cite{AB24} for the Lubin--Tate case, see in particular Corollary 3.4.10 of ibid (the hypothesis ``$\tau : G_L \to o_L^\times$ is surjective'' is replaced here by ``$G_K$ acts transitively on $\tphx$'', cf. Lemma \ref{galsurj} below). To keep this paper short, we focus on proving Theorem B. 

Besides clarifying the arguments of \S 3 of \cite{AB24}, our proof also shows that $\hatU(G)$ has a perfectoid-like nature. In particular, the injectivity of $\calK$ is related to the following result in $p$-adic Hodge theory (cf. 5.1.4 of \cite{Fon94}, and its Lubin--Tate generalizations such as Proposition 9.6 of \cite{Col02}): $\{ x \in \at^+$ such that $\theta \circ \phi^n(x)=0$ for all $n \geq 0\} = \pi \cdot \at^+$.


\subsection*{AI disclosure}  The first version of this paper was written in 2025 without AI. I used AI to check and correct the mathematical content in the revised version, as well as to improve the writing.

\section{Preliminaries}

In all this paper, we let $q=p^2$. The multiplication-by-$p$ map on $G$ is given by a power series $[p]_G(X)  = \sum_{i \geq 1} r_i X^i$ with $r_1=p$ and $r_i \in p \OO_K$ for $i \leq q-1$ and $r_q \in \OO_K^\times$. Recall that every element of $\OO_K \dcroc{X}$ that has a unit coefficient has a Weierstrass factorization, namely that it is the product of a distinguished polynomial and a unit.

\begin{lemm}
\label{galsurj}
If $t \in \tphx$, the extension $K(t_n)/K$ is totally ramified of degree $q^{n-1}(q-1)$ for all $n \geq 1$, and if
$t,t' \in \tphx$, there exists $\sigma \in G_K$ such that $t' = \sigma(t)$.
\end{lemm}

\begin{proof}
In the Weierstrass factorization of $[p^n]_H(X) / [p^{n-1}]_H(X)$, the distinguished poly\-nomial is Eisenstein, as its constant term is $p$ times a unit, and $K/\Qp$ is unramified. It is therefore irreducible, so that its roots are conjugate by $G_K$. If $t,t' \in \tphx$, then $t_n$ and $t'_n$ are both roots of $[p^n]_H(X) / [p^{n-1}]_H(X)$. This implies the claims.
\end{proof}

\begin{lemm}
\label{chgvar}
There is a change of variables such that $[p]_G(X) = pX+ \bigO(X^q)$.
\end{lemm}

\begin{proof}
We have $[p]_G(X) = pX+ \sum_{i=2}^{q-1} r_i X^i + \bigO(X^q)$ with $r_i \in p \cdot \OO_K$ for all $i \leq q-1$. Suppose that $r_2,\hdots,r_{m-1} = 0$. Conjugating by $X+\lambda \cdot X^m$ changes $r_m$ to $r_m + \lambda \cdot (p-p^m)$. Hence one can prove the claim by induction by taking $\lambda = r_m/(p^m-p)$. 
\end{proof}

A coordinate satisfying the above conditions is said to be clean. Let $\log_G(X) \in K \dcroc{X}$ denote the logarithm attached to the formal group $G$. The following result is clear.

\begin{lemm}
\label{logspec}
If $X$ is clean, then $\log_G(X) = X + \alpha \cdot X^q/p + \bigO(X^{q+1})$ for some $\alpha \in \OO_K^\times$. 
\end{lemm}

Let $t$ be an element of $\tph \simeq \Hom_{\OO_{\Cp}}(G,\Gm)$, and let $t(X) = 1 + \sum_{i \geq 1} a_i(t) X^i \in \OO_{\Cp} \dcroc{X}$ be the corresponding homomorphism from $G$ to $\Gm$. It is given by $t(X) = \exp( a_1(t) \cdot \log_G(X) )$ (and is denoted by $\Delta_1(Z)$ in \cite{AB24} in the Lubin--Tate case).

\begin{lemm}
\label{galth}
If $\sigma \in G_K$ then $a_n(\sigma(t)) = \sigma(a_n(t))$ for all $n \geq 1$.
\end{lemm}

\begin{proof}
The isomorphism $\tph \simeq \Hom_{\OO_{\Cp}}(G,\Gm)$ is $G_K$-equivariant.
\end{proof}

\begin{lemm}
\label{acoeff}
We have $a_n(t) = a_1(t)^n/n!$ if $n \leq q-1$ and $a_q(t) = a_1(t)^q/q! + \alpha \cdot a_1(t)/p$.
\end{lemm}

\begin{proof}
Since $t(X) = \exp( a_1(t) \cdot \log_G(X) )$, the claim follows from Lemma \ref{logspec}.
\end{proof}

\begin{prop}
\label{katzval}
If $t \in \tphx$,  the abscissae of the breakpoints of the Newton polygon of $t(X)-1$ are the $p^m$ with $m \geq 0$, and $\vp(a_{p^m}(t)) = 1/p^{m-1}(q-1)$ for all $m \geq 0$.
\end{prop}

\begin{proof}
The Newton polygon of $t(X)-1$ is independent of the choice of coordinate $X$, and we choose a clean coordinate on $G$.
We first prove that $\vp(a_1(t)) \leq p/(q-1)$. Let $\zeta$ be a primitive $p$-th root of unity. Since $t \in \tphx$, there exists $\eta \in G[p]$ such that $t(\eta)=\zeta$. We have $\vp(\eta)=1/(q-1)$ and $\vp(\zeta-1)=1/(p-1)$. 
Since $X$ is a clean coordinate on $G$, we have $a_n(t) = a_1(t)^n/n!$ if $n \leq q-1$ by Lemma \ref{acoeff}.
If $\vp(a_1) > p/(q-1)$, then for $n < q$
\[ \vp(a_n(t) \eta^n) = \vp\left(a_1(t)^n \frac{\eta^n}{n!}\right)  > \frac{np+n}{q-1}-\frac{n-s_p(n)}{p-1} = \frac{s_p(n)}{p-1}  \geq \frac{1}{p-1}, \] 
while $\vp(\eta^n) = n/(q-1) > 1$ if $n \geq q$. It is therefore not possible to have $t(\eta)-1=\zeta-1$, so that $\vp(a_1(t)) \leq p/(q-1)$.

For every $m \geq 1$, $t(X)-1$ has $p^m-p^{m-1}$ zeroes in $G[p^m] \setminus G[p^{m-1}]$ and they are all of valuation $1/(q^m-q^{m-1})$. Since $\sum_{m \geq 1} (p^m-p^{m-1}) / (q^m-q^{m-1}) = p/(q-1)$, the theory of Newton polygons tells us that $\vp(a_1(t)) = p/(q-1)$, and then that $t(X)-1$ cannot have other zeroes. 

The valuations of the $a_{p^m}(t)$ can then be read off the Newton polygon of $t(X)-1$.
\end{proof}

\begin{rema}
\label{kalm2}
{\hspace{1pt}}
\begin{enumerate}
\item Proposition \ref{katzval} applied to the Lubin--Tate case is Corollary 3.5.8 of \cite{AB24};
\item The proof of Proposition \ref{katzval} is inspired by the proof of Lemma 13 of \cite{Box86};
\item Proposition \ref{katzval} applied to $m=0$ gives us Lemma 1 on page 62 of \cite{Kat77};
\item Lemma \ref{acoeff} and the fact that $\vp(a_q(t)) = 1/p(q-1)$ imply that 
\[ \vp\left(a_1(t) \cdot \left(\frac{a_1(t)^{q-1}}{q!} + \frac{\alpha}{p} \right)\right) = \frac {1}{p(q-1)} \] and hence
\[ \vp\left(\frac{a_1(t)^{q-1} \cdot p}{\alpha \cdot q!} +1\right) = 1 - \frac 1p. \]
Compare with Lemma 2 on page 63 of \cite{Kat77}, using $\vp(q!) = p+1$. 
\end{enumerate}
\end{rema}

\section{The Katz map}

Let $C^0_{\Gal}(\tph, \OO_{\Cp})$ denote the set of Galois continuous functions $f : \tph \to \OO_{\Cp}$, namely those maps that are continuous and $G_K$-equivariant: $f(\sigma(t)) = \sigma(f(t))$ for all $t \in \tph$ and $\sigma \in G_K$. By Lemma \ref{galsurj}, any two elements of $p^n \cdot \tphx$ are Galois conjugates, so if we fix $t \in \tphx$, a function $f \in C^0_{\Gal}(\tph, \OO_{\Cp})$ is determined by $\{ f(p^n t) \}_{n \geq 0}$. In addition, we have $f(p^n t) \to f(0)$ as $n \to +\infty$ by continuity, and $f(0) \in \OO_K$ by the Ax-Sen-Tate theorem since $f(0)$ is fixed by $G_K$. 

Write $t = (t_0,t_1, \hdots) \in \tph$ and let $K_n$ denote the extension of $K$ generated by $t_n$ and let $K_\infty$ be the completion of $\cup_{n \geq 0} K_n$. By the Ax-Sen-Tate theorem, we have $f(p^n t) \in \OO_{K_\infty}$ for all $n$. Let $\prod_{n \geq 0}' \OO_{K_\infty}$ denote the $\OO_K$-algebra of sequences $\{ f_n \}_{n \geq 0}$ with $f_n\in \OO_{K_\infty}$ and such that $\{ f_n \}_{n \geq 0}$ converges to an element of $\OO_K$ as $n \to +\infty$. 
The map $f \mapsto \{ f(p^n t) \}_{n \geq 0}$ gives an isomorphism $C^0_{\Gal}(\tph, \OO_{\Cp}) \to \prod_{n \geq 0}' \OO_{K_\infty}$. 

Recall that $U(G)$ is the covariant bialgebra of $G$, namely the set of all $\OO_K$-linear maps $A(G) \to \OO_K$ that vanish on some power of the augmentation ideal. If $f,g \in U(G)$, their product is defined by $(f \cdot g)(a(X)) = (f \otimes g)a(X \oplus_G Y)$.  

\begin{rema}
\label{ugdiff}
In \cite{Kat77} and \cite{Kat81}, Katz introduces the algebra $\Diff(G)$ of translation-invariant differential operators on $G$. We do not use this object, but note the following for the record. If $D \in \Diff(G)$, then $[ f \mapsto D(f)(0) ] \in U(G)$, and this map gives rise to an isomorphism between the $\OO_K$-algebras $\Diff(G)$ and $U(G)$.

The proof is the same as that in \S 3.2 of \cite{AB24} for the Lubin--Tate case: for $n \geq 0$, let $u_n \in U(G)$ be the map $u_n : A(G) \to \OO_K$ given by $u_n(\sum_{i \geq 0} b_i X^i)=b_n$. We have $u_n \cdot u_m = \sum_{k \geq 0} s_{k,n,m} u_k$ where $(X \oplus_G Y)^k = \sum_{n,m \geq 0} s_{k,n,m} X^n Y^m$ for $k \geq 0$, and we get the same structure constants as in $\Diff(G)$, see for instance (1.2) of \cite{Kat81}.
\end{rema}

Recall that $\hatU(G)$ denotes the set of $\OO_K$-linear maps $A(G) \to \OO_K$ that are continuous for the $(p,X)$-adic topology, so that $\hatU(G)$ is the $p$-adic completion of $U(G)$. The evaluation pairing $\pscal{\cdot,\cdot} : \hatU(G) \times A(G) \to \OO_K$  extends to $\pscal{\cdot,\cdot} : \hatU(G) \times (\OO_{\Cp} \wotimes A(G)) \to \OO_{\Cp}$.

\begin{defi}
\label{defkatz}
The Katz map is the map $\calK : \hatU(G) \to C^0_{\Gal}(\tph, \OO_{\Cp})$ defined by $\calK(u) (t) = \pscal{u, t(X)}$. 
\end{defi}

The map $\calK(u) : \tph \to \OO_{\Cp}$ is Galois-continuous by Lemma \ref{galth}. The Katz map $\calK$ is an $\OO_K$-algebra homomorphism (the proof is the same as that of Lemma 3.3.5 of \cite{AB24}). Let $\phi_C : C^0_{\Gal}(\tph, \OO_{\Cp}) \to C^0_{\Gal}(\tph, \OO_{\Cp})$ denote the map given by $(\phi_C f)(t) = f(pt)$. We have a map $U(\phi_G) : \hatU(G) \to \hatU(G)$ coming by duality from the map $\phi_G : A(G) \to A(G)$ given by $a(X) \mapsto a([p]_G(X))$.

\begin{lemm}
\label{phikatz}
We have $\calK \circ U(\phi_G) = \phi_C \circ \calK$.
\end{lemm}

\begin{proof}
The proof is the same as that of Lemma 3.3.6 of \cite{AB24}, where $U(\phi_G)$ is denoted by $\phi^*$.
\end{proof}

\begin{rema}
\label{bprime}
Since $\hatU(G) = \{ \sum_{n \geq 0} \lambda_n u_n$ with $\{ \lambda_n\}_{n \geq 0} \in c^0(\OO_K) \}$, and $u_n(t(X))= a_n(t)$,
we can reformulate Theorem B as follows:
\begin{enumerate}
\item If $\{ \lambda_n\} \in c^0(\OO_K)$ and $\sum_{n \geq 0} \lambda_n a_n(t) = 0$ for all $t \in \tph$, then $\lambda_n = 0$ for all $n$;
\item If $t \in \tphx$,  every $x \in \OO_{K_\infty}$ can be written as $x = \sum_{n \geq 0} \lambda_n a_n(t)$ with $\{ \lambda_n\} \in c^0(\OO_K)$.
\end{enumerate}
\end{rema}

\section{Surjectivity of $\calK_t$}

We now prove that $\calK_t : \hatU(G) \to \OO_{K_\infty}$ is surjective if $t = (t_0,t_1, \hdots) \in \tphx$. 

For $n \geq 0$, let $G_n=G[p^n]$ and $H_n=H[p^n]$ and recall that $K_n = K(t_n)$. The inclusion $G_n \to G_{n+1}$ is (\S 2.3 of \cite{Tat67}) the Cartier dual of $[p]_H : H_{n+1} \to H_n$ (and vice versa). We have $A(G_n) = \OO_K\dcroc{X}/\phi_G^n(X)$. Let $U(G_n) = \Hom_{\OO_K}(A(G_n),\OO_K) \subset \hatU(G)$. Cartier duality gives an isomorphism $U(G_n) = A(H_n)$, so that $U(G_n) = \OO_K\dcroc{Y_n}/\phi_H^n(Y_n)$. On $U(G_n)$, we have $U([p]_G) = [p]_H$ so that $U(\phi_G) = \phi_H$. The natural inclusion $U(G_n) \to U(G_{n+1})$ is the map $A(H_n) \to A(H_{n+1})$ that comes from $[p]_H : H_{n+1} \to H_n$. Its image is $\phi_H(U(G_{n+1}))$. 
 If $U(G_n)_K = K \otimes_{\OO_K} U(G_n)$, then 
\[ U(G_n)_K =  K \otimes_{\OO_K} \OO_K\dcroc{Y_n}/\phi_H^n(Y_n) \simeq K_n \times K_{n-1} \times \cdots \times K_1 \times K_0. \] 

Fix $t \in \tphx$. For $n \geq 0$, let $\kappa_n : \hatU(G) \to \OO_{K_\infty}$ denote the map $u \mapsto \calK(u)(p^n t)$. Lemma \ref{phikatz} implies that $\kappa_n(\phi_H(u)) = \kappa_{n+1}(u)$ if $u \in U(G_m)$ for some $m, n \geq 0$.

\begin{prop}
\label{katsurj}
For all $n \geq 1$, the image of the map $\kappa_0 : U(G_n)_K \to K_\infty$ is $K_n$.
\end{prop}

\begin{proof}
We first prove that the image of $\kappa_0 : U(G_n)_K \to K_\infty$ is contained in $K_n$. By Cartier duality, we have $H_n = \Hom_{\OO_{\Cp}}(G_n,\Gm)$, and the map $t_n(X) \in \Hom_{\OO_{\Cp}}(G_n,\Gm)$ corresponding to $t_n$ is given by $t(X) \bmod{\phi_G^n(X)}$. We then have 
\[ t_n(X) \in \OO_{\Cp} \otimes A(G_n) = \Hom_{\OO_{\Cp}} ( \OO_{\Cp} \otimes U(G_n) , \OO_{\Cp}) \]
The resulting map $U(G_n)_K \to \Cp$ is given by $u \mapsto u(t(X) \bmod{\phi_G^n(X)}) = u(t)$, namely the restriction of $\kappa_0$ to $U(G_n)_K$. For all $\sigma \in G_{K_n}$, $\sigma(t_n) = t_n$ and hence $\sigma(t_n(X)) = t_n(X)$ so that $\sigma(\kappa_0(u)) = \kappa_0(u)$ if $u \in U(G_n)_K$. By the Ax-Sen-Tate theorem, $\kappa_0(u) \in K_n$.

We now prove that the image of $\kappa_0$ is $K_n$. Let $[p^n]_G(X) = f(X) v(X)$ be the Weierstrass factorization of $[p^n]_G(X)$, where $f(X)$ is a distinguished polynomial of degree $q^n$ and $v(X)$ is a unit. Thanks to Weierstrass division by $f(X)$, we can write 
\[ \OO_K \dcroc{X} = (\oplus_{i=0}^{q^n-1} \OO_K X^i) \oplus f(X) \OO_K \dcroc{X}. \]
For $0 \leq r \leq q^n-1$, let $w_r \in U(G_n)_K$ be the $\OO_K$-linear form that maps $X^i$ to $\delta_{ir}$ for $i \leq q^n-1$ and maps $f(X) \OO_K \dcroc{X}$ to $0$. It maps $X^{j+q^n}$ to $p \OO_K$ for all $j \geq 0$ since $f(X) = X^{q^n} +p g(X)$ and hence $X^{j+q^n} = X^j f(X) - p X^j g(X)$. 

By Proposition \ref{katzval}, we have $\vp(a_{p^m}(t)) = 1/p^{m-1}(q-1)$ for all $m \geq 0$. In particular, setting $m=2n-1$, so that $p^m=q^n/p \leq q^n-1$, we have  $\vp(a_{p^m}(t)) = 1/q^{n-1}(q-1)$. We have \[ \kappa_0(w_{p^{2n-1}}) = w_{p^{2n-1}}(t(X)) \equiv a_{p^m}(t) \bmod{p}, \] so that $\kappa_0(U(G_n)_K)$ contains a uniformizer of $K_n$. It is therefore equal to $K_n$ since $K_n/K$ is totally ramified.
\end{proof}

\begin{rema}
\label{surjab}
Compare with (the proof of) Proposition 3.6.7 of \cite{AB24}.
\end{rema}

Recall that we have an isomorphism $U(G_n) = \OO_K\dcroc{Y_n}/\phi_H^n(Y_n)$.

\begin{prop}
\label{katzum}
If $n \geq 1$ , there exists $u_0,u_1,\hdots,u_n \in K_\infty$ such that
\begin{enumerate}
\item $u_0=0$ and $u_1 \neq 0$;
\item $[p]_H(u_k) = u_{k-1}$ for $1 \leq k \leq n$;
\item the map $\kappa_j : U(G_n)_K \to K_\infty$ is given by $P(Y_n) \mapsto P(u_{n-j})$ if $j \leq n$.
\end{enumerate}
In addition, the map $\kappa_j : U(G_n)_K \to K_\infty$ is given by $P(Y_n) \mapsto P(0)$ if $j \geq n$.
\end{prop}

\begin{proof}
For all $0 \leq j \leq n$, there exists a root $u_{n-j}$ of $\phi_H^n(Y_n)$ in $\MM_{K_\infty}$ such that the $K$-algebra map $\kappa_j : K \otimes_{\OO_K} \OO_K\dcroc{Y_n}/\phi_H^n(Y_n) \to K_\infty$ is given by $P(Y_n) \mapsto P(u_{n-j})$. In addition, we have $\kappa_{j-1} ( \phi_H(Y_n) ) = \kappa_j(Y_n)$ so that $[p]_H(u_{n-j+1}) = u_{n-j}$. 

We have $\kappa_n(Y_n) = \kappa_0 ( \phi_H^n(Y_n) ) = 0$ so that $u_0=0$, and finally since $\kappa_0 : U(G_1)_K \to K_1$ is surjective by Proposition \ref{katsurj}, and 
\[ \kappa_{n-1}(U(G_n)_K) = \kappa_0(\phi_H^{n-1}(U(G_n))_K) = \kappa_0(U(G_1)_K), \]
the element $u_1$ generates $K_1$ over $K$ and so $u_1 \neq 0$. 

Finally, if $j \geq n$, then $\kappa_j(Y_n) = \kappa_0(\phi_H^j(Y_n)) = 0$.
\end{proof}

\begin{coro}
\label{katsurjint}
The map $\kappa_0 : U(G_n) \to \OO_{K_n}$ is surjective for all $n \geq 0$.
\end{coro}

\begin{proof}
The map $\kappa_0 : U(G_n) \to \OO_{K_\infty}$ is given by $P(Y_n) \mapsto P(u_n)$ as in Proposition \ref{katzum}, so that its image is $\OO_{K_n}$ since $u_n$ is a uniformizer of $K_n$.
\end{proof}

\begin{coro}
\label{kappasurj}
The map $\kappa_0 : \hatU(G) \to \OO_{K_\infty}$ is surjective.
\end{coro}

This proves the surjectivity part of Theorem B since $\kappa_0 = \calK_t$.

\section{Injectivity of $\calK$}

We now prove the injectivity of the Katz map $\calK : \hatU(G) \to C^0_{\Gal}(\tph,\OO_{\Cp})$. Recall that $U(G_n) = \Hom_{\OO_K}(\OO_K\dcroc{X}/\phi_G^n(X),\OO_K)$ and that we have an injection $U(G_n) \to \hatU(G)$.

\begin{lemm}
\label{undense}
The $\OO_K$-module $\cup_{n \geq 1} U(G_n)$ is $p$-adically dense in $\hatU(G)$.
\end{lemm}

\begin{proof}
Take $k,d \geq 0$ and $u \in \hatU(G)$ such that $u(X^d \OO_K \dcroc{X}) \subset p^k \OO_K$, and  $n \geq 1$ such that $[p^n]_G(X) \in (X^d,p^k)$. Let $[p^n]_G(X) = f(X) v(X)$ be the Weierstrass factorization of $[p^n]_G(X)$, where $f(X)$ is distinguished of degree $q^n$ and $v(X)$ is a unit. Write $\OO_K \dcroc{X} = (\oplus_{i=0}^{q^n-1} \OO_K X^i) \oplus f(X) \OO_K \dcroc{X}$. Define $w \in U(G_n)$ by $w=u$ on $1,X,\hdots,X^{q^n-1}$ and $w=0$ on $f(X) \OO_K \dcroc{X}$. If $j \geq 0$, then $X^{j+q^n} \in f(X) \OO_K \dcroc{X} +(\oplus_{i=d+j}^{q^n+j-1} \OO_K X^i)  + p^k \OO_K \dcroc{X}$. By induction on $j$, we have $w(X^{j+q^n}) \in p^k \OO_K$ for all $j \geq 0$ and therefore $w-u \in p^k \hatU(G)$.
\end{proof}

The isomorphisms $U(G_n) = \OO_K\dcroc{Y_n}/\phi_H^n(Y_n)$ are compatible as $n$ varies, so that we get a common $u =(\hdots,u_1,u_0)$ for all $U(G_n)_K$ in Proposition \ref{katzum}. We have $\phi_H(Y_{n+1}) = Y_n$ for all $n$, and we now write $\phi_H^{-n}(Y)$ instead of $Y_n$. We can then write $U(G_n) = \OO_K \dcroc{ \phi_H^{-n}(Y) } / Y$ for each $n$, with each transition map $U(G_{n-1}) \to U(G_n)$ sending $Y$ to $Y$ and the map $\kappa_j$ sending $\phi_H^{-n}(Y)$ to $u_{n-j}$. Lemma \ref{undense} and the fact that $U(G_n) \cap p \cdot \hatU(G) = p \cdot U(G_n)$ for all $n \geq 0$ imply that $\hatU(G)$ is the $p$-adic completion of $\cup_{n \geq 0} \OO_K \dcroc { \phi_H^{-n}(Y) } / Y$.

Let $\afont$ be the $p$-adic completion of $\cup_{n \geq 0} \OO_K \dcroc { \phi_H^{-n}(Y) }$ and let $\theta : \afont \to \OO_{K_\infty}$ be the $\OO_K$-linear ring homomorphism that sends $\phi_H^{-n}(Y)$ to $u_n$ so that $\hatU(G) = \afont/Y$ and $\kappa_0 = \theta$. 

Let $\efont$ denote the ring $\cup_{n \geq 0} k \dcroc{ \phi_H^{-n}(Y) }$. The valuation $\val_Y$ is compatibly defined on each $k \dcroc{ \phi_H^{-n}(Y) }$ and hence on $\efont$ (indeed, $\phi_H(Y) \in Y^q \cdot k \dcroc{Y}^\times$ so that if $x \in k \dcroc{Y_n} \subset  k \dcroc{Y_{n+1}}$ with $Y_n = \phi_H(Y_{n+1})$, then $\val_{Y_{n+1}}(x) = q \cdot \val_{Y_n}(x)$). The following property is clear.

\begin{lemm}
\label{divy}
If $x \in \efont$, then $x \in Y \cdot \efont$ if and only if $\val_Y(x) \geq 1$.
\end{lemm}

Let $\theta : \efont \to \OO_{K_\infty} / p$ be the $k$-linear ring homomorphism that sends $\phi_H^{-n}(Y)$ to $u_n$ so that $\efont = \afont / p \afont$ with compatible $\theta$.

\begin{lemm}
\label{kerthetet}
We have $\ker(\theta : \efont \to \OO_{K_\infty} / p) = Y/\phi_H^{-1}(Y) \cdot \efont$. 
\end{lemm}

\begin{proof}
It is enough to prove that for all $n \geq 1$, $\ker(\theta : k \dcroc{ \phi_H^{-n}(Y) } \to \OO_{K_\infty} / p)$ is generated by $Y/\phi_H^{-1}(Y)$. Let $Q_n(X)$ be the minimal polynomial of $u_n$ over $K$, so that $Q_n(X) \in \OO_K[X]$ is monic of degree $q^{n-1}(q-1)$. If $P(X) \in \OO_K \dcroc{X}$, we can write it as $P=SQ_n+R$ with $\deg R < q^{n-1}(q-1)$. If $P(u_n) \in p \cdot \OO_{K_n}$, then $R(u_n) \in p \cdot \OO_{K_n}$ and since $\vp(u_n) = 1/q^{n-1}(q-1)$, this implies that $R(X) \in p \cdot \OO_K[X]$. 
The claim now follows from this, and from the fact that $\theta(\phi_H^{-n}(Y)) = u_n$ so that $\ker(\theta)$ is generated by $\phi_H^{-n}(Y)^{q^{n-1}(q-1)}$ and hence by $Y/\phi_H^{-1}(Y)$ since $\phi_H(Y) = Y^q \cdot h(Y)$ with $h(Y) \in k \dcroc{Y}^\times$.
\end{proof}

\begin{prop}
\label{kerthet}
We have $\ker(\theta : \afont \to \OO_{K_\infty}) = Y/\phi_H^{-1}(Y) \cdot \afont$. 
\end{prop}

\begin{proof}
If $x \in \ker(\theta)$, then $\overline{x} \in \efont$ is also killed by $\theta$ and is  divisible by $Y/\phi_H^{-1}(Y)$ in $\efont$ by Lemma \ref{kerthetet}. We can therefore write $x= Y/\phi_H^{-1}(Y) \cdot x_1 + p y_1$ with $x_1,y_1 \in \afont$ and $\theta(y_1)=0$. By induction, we can write $x= Y/\phi_H^{-1}(Y) \cdot x_k + p^k y_k$ with $x_k,y_k \in \afont$ and $\theta(y_k)=0$ for all $k \geq 1$. Since $\afont$ is $p$-adically complete, this implies the claim.
\end{proof}

\begin{prop}
\label{kerkatz}
We have $\{ x \in \afont$, $\theta \circ \phi_H^n(x)=0$ for all $n \geq 0\} = Y \cdot \afont$.
\end{prop}

\begin{proof}
One inclusion is clear, since $\theta \circ \phi_H^n(Y)=0$ for all $n \geq 0$. We now prove the reverse inclusion. Proposition \ref{kerthet} shows that $\ker(\theta) = Y/\phi_H^{-1}(Y) \cdot \afont$ and therefore that for all $j \geq 0$, $\ker(\theta \circ \phi_H^j) = \phi_H^{-j}(Y)/\phi_H^{-j-1}(Y) \cdot \afont$. For $n \geq 1$, let $I_n$ denote the set of $x \in \afont$ such that $(\theta \circ \phi_H^j) (x) = 0$ for $0 \leq j \leq n-1$. Since $(\theta \circ \phi_H^\ell)(\phi_H^{-j}(Y)/\phi_H^{-j-1}(Y)) \neq 0$ if $\ell < j$, and $Y/\phi_H^{-1}(Y) \cdot \phi_H^{-1}(Y)/\phi_H^{-2}(Y) \cdots \phi_H^{-(n-1)}(Y)/\phi_H^{-n}(Y) = Y/\phi_H^{-n}(Y)$, we have $I_n =  Y/\phi_H^{-n}(Y) \cdot \afont$. Let $I = \cap_{n \geq 1} I_n = \{ x \in \afont$, $\theta \circ \phi_H^j(x)=0$ for all $j \geq 0\}$.

The above reasoning and Lemma \ref{divy} imply that in $\efont$, we have $\overline{I}= Y \cdot \efont$. Hence if $x \in I$, then $\overline{x} \in Y \cdot \efont$. We can therefore write $x= Y \cdot x_1 + p y_1$ with $x_1 \in \afont$ and $y_1 \in I$. By induction, we can write $x= Y \cdot x_k + p^k y_k$ with $x_k \in \afont$ and $y_k \in I$ for all $k \geq 1$. Since $\afont$ is $p$-adically complete, this implies the claim.
\end{proof}

\begin{rema}
\label{colthet}
Compare with Proposition 9.6 of \cite{Col02} (and 5.1.4 of \cite{Fon94}), noting however that our $\efont$ is not complete for the $Y$-adic topology. We could actually replace $\efont$ and $\afont$ above with their $Y$-adic completions, since in any case $\hatU(G) = \afont/Y$. In the Lubin--Tate case, we would then have $\efont = \et^+_K$ and $\afont = \at^+_K$ in the notation of ibid.

In general, $\phi_H(Y) = h(Y^q)$ in $k \dcroc{Y}$ for some reversible $h(Y)$ so that $\efont$ is still perfect.
\end{rema}

\begin{coro}
\label{katzinj}
The map $\calK : \hatU(G) \to C^0_{\Gal}(\tph,\OO_{\Cp})$ is injective.
\end{coro}

\begin{proof}
If we write $C^0_{\Gal}(\tph,\OO_{\Cp}) = \prod'_{n \geq 0} \OO_{K_\infty}$, then $\calK : \hatU(G) \to C^0_{\Gal}(\tph,\OO_{\Cp})$ comes from $\{ \theta \circ \phi_H^n \}_{n \geq 0} : \afont \to \prod'_{n \geq 0} \OO_{K_\infty}$ and the claim results from Proposition \ref{kerkatz} above and the fact that $\hatU(G) = \afont/Y$ by Lemma \ref{undense}.
\end{proof}

This finishes the proof of Theorem B.

\end{document}